\newif\ifpictures
\picturestrue

\documentclass[12pt]{amsart}
\usepackage{amssymb,amsmath}
 \usepackage{amsopn}
 \usepackage{xspace}
  \usepackage{hyperref}
 \usepackage[dvips]{graphicx}

\numberwithin{equation}{section}
\newtheorem{thm}{Theorem}
\newtheorem{prop}[thm]{Proposition}
\newtheorem{lemma}[thm]{Lemma}
\newtheorem{cor}[thm]{Corollary}

\newtheorem{definition}[thm]{Definition}

\numberwithin{thm}{section}

\newcounter{FNC}[page]
\def\newfootnote#1{{\addtocounter{FNC}{2}$^\fnsymbol{FNC}$%
     \let\thefootnote\relax\footnotetext{$^\fnsymbol{FNC}$#1}}}

\newcommand{\N}{\mathbb{N}}

\newcommand{\R}{\mathbb{R}}

\newcommand{\Z}{\mathbb{Z}}

\newcommand{\sig}{\sigma}

\DeclareMathOperator{\sign}{sign}
\DeclareMathOperator{\rank}{rank}
\DeclareMathOperator{\Jac}{Jac}

\title[Low Dimensional Test Sets for Nonnegativity of Even Symmetric Forms]{Low Dimensional Test Sets for Nonnegativity of Even Symmetric Forms}
\author{Sadik Iliman} \author{Timo de Wolff}

\address{Goethe-Universit\"at, FB 12 -- Institut f\"ur Mathematik,
Postfach 11 19 32, D-60054 Frankfurt am Main, Germany}

\email{\{iliman,wolff\}@math.uni-frankfurt.de}


\begin{document}

\begin{abstract}
 An important theorem by Timofte states that nonnegativity of real $n$-variate symmetric polynomials of degree $d$ can be decided at test sets given by all points with at most $\lfloor\frac{d}{2}\rfloor$ distinct components. However, if the degree is sufficiently larger than the number of variables, then the theorem obviously does not provide nontrivial information. Our approach is to look at $(m + 1)$-dimensional subspaces of even symmetric forms of degree $4d$, at which nonnegativity can be checked at $(m - 1)$-points, i.e., points with at most $m - 1 \in \N$ distinct components, where $m$ is independent of the degree of the forms and better than Timofte's bound. Furthermore, for fixed $k \in \N$, we tackle problems concerning the maximum dimension of such subspaces, at which nonnegativity can be checked at all $k$-points, as well as the geometrical and topological structure of the set of all forms whose nonnegativity can be decided at all $k$-points.

\end{abstract}

\maketitle

\section{Introduction}

The theory of the cones of nonnegative polynomials and sums of squares is very crucial for many theoretical and practical problems in convex algebraic geometry (see, e.g., \cite{Blekherman:Parrilo:Thomas,Lasserre}). By Hilbert's theorem in \cite{Hilbert:Seminal} these two cones coincide exactly for binary forms $(n = 2)$, quadratic forms $(2d = 2)$ and ternary quartics $(n = 3, 2d = 4)$. In contrast to deciding whether a polynomial is a sum of squares, the problem of deciding nonnegativity of polynomials is an NP-hard problem (see, e.g., \cite{Blum:et:al:np}). Furthermore, it is known that for fixed degree $2d \geq 4$ and growing number of variables there are significantly more nonnegative polynomials than sums of squares (\cite{Blekherman:Volume}).
 Therefore, a convincing alternative approach in order to simplify the question whether a real polynomial $p$ of even degree $2d$ is nonnegative, is to classify  \textit{test sets} $\Omega \subset \mathbb R^n$ for nonnegativity of polynomials in order to reduce the complexity of deciding nonnegativity. Here, we call $\Omega \subset  \R^n$ a test set if $p(x) \geq 0$ for all $x \in \mathbb R^n$ if and only if $p(x)\geq 0$ for all $x \in \Omega$. For example, if $p$ is a homogeneous polynomial, then $\Omega := \mathbb S^{n-1}$ is a test set, but it does not reduce the complexity of deciding nonnegativity of homogeneous polynomials. Although for arbitrary polynomials test sets reducing the complexity of deciding nonnegativity are unknown as well as seemingly difficult to find, for \textit{symmetric} polynomials such test sets exist.

The problem of constructing test sets for symmetric forms began with the work of Choi, Lam, Reznick in \cite{Choi:Lam:Reznick}, in which the authors considered test sets for even symmetric sextics and were able to give a complete
semialgebraic characterization of nonnegative even symmetric sextics and even symmetric sextics that are sums of squares. The key result is that checking nonnegativity in this case can be reduced
to checking nonnegativity of univariate polynomials since it suffices to prove nonnegativity of even symmetric sextics at all points in $\mathbb R^n$ with at most one nonzero
component. Later, Harris (\cite{Harris}; see also \cite{Harris:Diss}) generalized some results by establishing that even symmetric octics (degree $2d = 8$) are nonnegative if and only if they are nonnegative at all points in $\mathbb R^n$ with at most two nonzero components. Indeed, in the case of even symmetric ternary octics he showed that every such nonnegative form 
is a sum of squares. Additionally, he proved that nonnegativity of even symmetric ternary decics ($n = 3, 2d = 10$) can be decided at points with at most two nonzero components, too. However, he also proved that nonnegativity of even symmetric ternary forms of degree $2d \geq 12$ cannot be checked by considering points with at most two nonzero components.

In \cite{Timofte} Timofte proved a very powerful result, namely that a symmetric polynomial of degree $d$ is nonnegative if and only if it is nonnegative at all points with at most $\lfloor\frac{d}{2}\rfloor$ distinct components. Later, Riener was able to reprove this result in a much more elementary setting than in the original work, where most techniques are based on the theory of differential equations (see \cite{Riener}; see also \cite{Riener:Diss}). For further results concerning nonnegativity of symmetric polynomials see, e.g., \cite{Blekherman:Riener,Gatermann:Parrilo}.\\

In this paper we are interested in even symmetric homogeneous polynomials (forms). We consider the question how to identify test sets for such forms and investigate their properties. In particular, we analyze under which additional conditions on even symmetric forms the bound of at most $\lfloor\frac{d}{2}\rfloor$ distinct components given by Timofte's theorem can be further improved.
Polynomials with such interesting structure are those lying in certain subspaces. We analyze the question whether it is even possible that there exist uniform bounds better than Timofte's one and independent of the degree of the polynomials.

We prove existence of such uniform bounds at certain subspaces of forms of degree $4d$. As a base case we extend results in \cite{Harris} and look at $4$-dimensional subspaces of forms given as
\begin{eqnarray*}
	p & := & \alpha M_{j_1}^{k_1} \cdots M_{j_r}^{k_r} + \beta M_2^{2d} + \gamma M_{2d}^2 + \delta M_{2d}M_2^d,
\end{eqnarray*}
where $\alpha,\beta,\gamma,\delta \in \R^*$, $M_j := \sum_{i=1}^n x_i^j$ is the $j$-th power sum polynomial and the following conditions are satisfied
\begin{eqnarray}
	& & j_1,\ldots,j_r \in 2\N, \ k_1,\ldots,k_r \in \N, \ \sum_{i = 1}^r j_i k_i = 4d, j_1 \notin \{2,2d\}, \label{Equ:ClassConditionsIntro} \\
	& & \text{and either } j_1,\ldots,j_r \leq 2d \text{ or } j_2,\ldots,j_r \in \{2,2d\}. \nonumber
\end{eqnarray}
Hence, the set of nonnegative forms of this type comprises a $4$-dimensional subcone 
of the cone of real even symmetric forms of degree $4d$. By adjusting the number of variables, we extend our approach to subspaces of arbitrary dimensions given by
\begin{eqnarray}
 p(x_1,\dots,x_n) & := & \sum_{i=1}^{m-2}\alpha_if_i(x) + \beta M_2^{2d} + \gamma M_{2d}^2 + \delta M_{2d}M_2^d, \label{Equ:OurClass2Intro}
\end{eqnarray}
where $\alpha_i, \beta, \gamma, \delta \in \mathbb R^*$, $m \leq n$ and $f_i(x)$ is a product of power sums $M_j$ of degree $4d$ satisfying some additional conditions, which can be considered as generalizations of the conditions \eqref{Equ:ClassConditionsIntro} (see also \eqref{Equ:ClassConditions2}).
For both types of subspaces we prove that the number of nonzero components, one has to check for nonnegativity of these forms, is bounded by two resp. $m - 1$ (Theorems \ref{Thm:Main} and \ref{Thm:Main2}). These bounds are often better and more useful as in Timofte's theorem, especially
when the degree $4d$ is significantly larger than the number of variables. Furthermore, for fixed $k \in \N^*$, we investigate the maximum dimension $m + 1$ such that, with slight abuse of notation, at all $(m + 1)$-subspaces given by a certain basis nonnegativity can be decided at all $k$-points. Additionally, we consider the set of all such forms where nonnegativity can be decided at all $k$-points.  In special cases we can explicitly determine these maximum dimensions (Theorem \ref{thm:maximum}) and prove that the corresponding set of all forms for which nonnegativity can be decided at points with at most two nonzero distinct components is not convex (Corollary \ref{cor:convex}). For convenience, we summarize our results in Theorem \ref{thm:mainmain}.

This paper is organized as follows: In Section \ref{Sec:Preliminaries} we provide some basic tools and definitions from the theory of symmetric polynomials. In Section \ref{Sec:testsets} we introduce test sets and present core problems on them we are interested in. Furthermore, we state Timofte's theorem and summarize our results (Theorem \ref{thm:mainmain}). In Section \ref{Sec:DimensionFour}
 we consider a $4$-dimensional subspace of even symmetric forms of degree $4d$ and prove our first main result. This is in contrast to the bound of Timofte's theorem in \cite{Timofte}.
Indeed, in general, it is not sufficient to investigate the set of all points with at most two  distinct nonzero components in order to prove nonnegativity of even symmetric forms of degree $4d \geq 12$ (see, e.g., \cite{Harris}). We end this section by applying our results on some examples and provide some conjectures based on these experiments. In Section \ref{Sec:arbitrary} we consider subspaces of arbitrary dimension given as in \eqref{Equ:OurClass2Intro}. We prove our second main result by adjusting the number of variables and generalizing techniques from Section \ref{Sec:DimensionFour}. In Section \ref{Sec:konstante} we tackle the problems concerning the maximum dimensions of the above subspaces resp. the geometrical and topological structure of the set of all forms whose nonnegativity can be decided at all $k$-points. Here, we prove Theorem \ref{thm:maximum} and Corollary \ref{cor:convex}. Finally, in Section $7$ we discuss some open problems.

\section{Preliminaries}
\label{Sec:Preliminaries}
In this section we introduce some notations and facts that are essential for upcoming results. We begin with some classical 
facts about symmetric polynomials. Let $\mathbb R[x_1,\dots,x_n]_d^S$ be the ring of symmetric polynomials of degree $d \in \mathbb N$. A homogeneous symmetric polynomial is
called \textit{even symmetric form} if all exponents are even. A vector 
$\lambda = (\lambda_1,\dots,\lambda_n)$ is called a \textit{partition} of $d$ if $\lambda_1\geq\lambda_2\geq\dots\geq\lambda_n\geq 0$ and $\lambda_1+\dots +\lambda_n = d$.
The dimension of $\mathbb R[x_1,\dots,x_n]_d^S$ is given by the number of partitions of $d$ with length at most $n$. Note that the dimension of $\mathbb R[x_1,\dots,x_n]_d^S$ is fixed, i.e., independent of $n$, whenever $n\geq d$. A fundamental theorem in the theory of symmetric polynomials states that every symmetric polynomial $p\in\mathbb R[x_1,\dots,x_n]_d^S$ can be written as a polynomial in the power sums $M_r(x) = \sum_{i=1}^n x_i^r, 0\leq r\leq n$. For an overview and introduction see, e.g., \cite{Sagan}. \\

We now introduce \textit{Schur polynomials} that also form a basis of $\mathbb R[x_1,\dots,x_n]_d^S$.
Let $d \neq 0$ be a natural number with a partition $d = \sum_{j = 1}^l d_j$, $d_1 \geq \cdots \geq d_l$ where all $d_j$ are positive integral numbers. For a fixed partition $(d_1,\ldots,d_l)$ the $l$\textit{-variate monomial symmetric function} $m_{(d_1,\ldots,d_l)}$ is given by
\begin{eqnarray*}
	m_{(d_1,\ldots,d_l)} & := & \sum_{\sig \in S_l} x_{\sig(1)}^{d_1} \cdots x_{\sig(n)}^{d_l},
\end{eqnarray*}
where $S_l$ is the symmetric group in $l$ elements and $\sig$ denotes a permutation in $S_l$.
We define 
\begin{eqnarray}
	D_{(d_1,\ldots,d_l)} & := & \det
	\left(\begin{array}{cccc}
		x_1^{d_1 + l -1} & x_1^{d_2 + l - 2} 	& \cdots & x_1^{d_l} \\
		\vdots		 & \vdots 		& \ddots & \vdots \\
		x_l^{d_1 + l -1} & x_l^{d_2 + l - 2} 	& \cdots & x_l^{d_l} \\
	\end{array}\right).
\label{Equ:Determinant}
\end{eqnarray}
Furthermore, we denote the determinant $\prod_{1 \leq i,j \leq l} (x_i - x_j)$ of the $(l \times l)$-Vandermonde Matrix by $\Delta_l$, i.e.,
\begin{eqnarray}
	\Delta_l & := & \det
	\left(\begin{array}{cccc}
		x_1^{l - 1} 	& x_1^{l - 2} 	& \cdots & 1 \\
		\vdots		& \vdots	& \ddots & \vdots \\
		x_l^{l - 1} 	& x_l^{l - 2} 	& \cdots & 1 \\
	\end{array}\right).
\label{Equ:Vandermonde}
\end{eqnarray}

The \textit{Schur function} $S_{(d_1,\ldots,d_l)}$ is defined as
\begin{eqnarray}
	S_{(d_1,\ldots,d_l)} & := & \frac{D_{(d_1,\ldots,d_l)}}{\Delta_l}.
\label{Equ:Schur}
\end{eqnarray}

It is a well known fact that Schur functions are, indeed, symmetric polynomials, which contain an amazing combinatorial structure. For example, the set of $l$-variate Schur polynomials form a basis of the ring of all symmetric $l$-variate polynomials and the monomials of the Schur polynomial $S_{(d_1,\ldots,d_l)}$ are in one to one correspondence to all semistandard $(d_1,\ldots,d_l)$-tableaux.
For our needs the following proposition is crucial (see, e.g., \cite{Sagan}).

\begin{prop}
Let $d_1 \geq \cdots \geq d_l \in \N$. The Schur polynomial $S_{(d_1,\ldots,d_l)}$ can be expressed as
\begin{eqnarray*}
	S_{(d_1,\ldots,d_l)} & = & \sum_{\{(c_1,\ldots,c_l) \in \N^l \ : \ \sum_{j = 1}^r c_j \leq \sum_{j = 1}^k d_j \text{ for all } 1 \leq r \leq l\}} \kappa_{(c_1,\ldots,c_l),(d_1,\ldots,d_l)} m_{(c_1,\ldots,c_l)},
\end{eqnarray*}
where all $\kappa_{(c_1,\ldots,c_l),(d_1,\ldots,d_l)}$ are nonnegative integers and all $m_{(c_1,\ldots,c_l)}$ are monomial symmetric functions.
\label{Prop:SchurPolynomial}
\end{prop}

Notice that the natural numbers $\kappa_{(c_1,\ldots,c_l),(d_1,\ldots,d_l)}$ are called \textit{Kostka numbers}. Combinatorially, $\kappa_{(c_1,\ldots,c_l),(d_1,\ldots,d_l)}$ equals the cardinality of the set of all semistandard $(c_1,\ldots,c_l)$-tableaux of type $(d_1,\ldots,d_l)$.\\

\section{The Structure of Test Sets}
\label{Sec:testsets}
In the following we provide a brief discussion of the structure of test sets. 
We define \textit{test sets} and $k$\textit{-points},
state Timofte's theorem, and set up the major notations and problems for the remainder of this work.

\begin{definition}
We say that a set $\Omega\subset \mathbb R^n$ is a \emph{test set} for $p \in \mathbb R[x_1,\dots,x_n]_d^S$ if the following does hold: $p\geq 0$ if and only if $p(x)\geq 0$ for all $x\in\Omega$.
\end{definition}

Known test sets for symmetric polynomials are always given by $k$-\textit{points}, i.e., by points with a bounded number of distinct components. For this, we define the following two sets.

\medskip
\begin{definition}
~
\begin{enumerate}
 \item Let $\Omega_k$ denote the set of all points $(x_1,\ldots,x_n) \in \R_{}^n$ such that there exist $a_1 < \ldots < a_k \in \R_{}$ with $x_i \in \{a_1,\ldots,a_k\}$ for every $1 \leq i \leq n$.
 \item Let $\Omega_k^+$ denote the set of all points $(x_1,\ldots,x_n) \in \R_{\geq 0}^n$ such that there exist $a_1 < \ldots < a_k \in \R_{>0}$ with $x_i \in \{0,a_1,\ldots,a_k\}$ for every $1 \leq i \leq n$. In this case we call a point $x\in\mathbb R^n$ a $k$-\emph{point}.
\end{enumerate}
\end{definition}

Timofte's theorem can be stated as follows.

\begin{thm}[Timofte \cite{Timofte}]
\label{thm:Timofte}
 Let $p\in\mathbb R[x_1,\dots,x_n]_{2d}^S$ with $d \geq 2$. Then 
\begin{enumerate}
 \item $\Omega_d$ is a test set for $p$.
\item  If $p$ is even symmetric, then $\Omega_{\lfloor \frac{d}{2}\rfloor}^+$ is a test set for $p$.
\end{enumerate}
\end{thm}

As an example, since we are dealing with even symmetric forms, the following does hold: Nonnegativity of even symmetric octics $(2d = 8)$ and even symmetric decics $(2d = 10)$ can be reduced to semidefinite feasibility problems since, by Timofte's theorem, one has to check whether these forms are nonnegative at all $2$-points. 
Hence, the problem reduces to check whether a finite number of binary forms are nonnegative. By Hilbert's theorem this can be decided by checking whether these
forms are sums of squares.\\

Our main goal is to characterize test sets based on $k$-points that are independent of the degree of the investigated forms (with an arbitrary fixed number of variables). Let $\mathbb R[x_1,\dots,x_n]_{4d}^{S,e}$ be the vector space of \emph{even} symmetric forms in $n$ variables of degree $4d$ and $B$ be the basis given by the power sum polynomials. In the following we always assume that $n \geq 3$ since the question of nonnegativity of binary forms is obvious by Hilbert's theorem. The key idea is to restrict to subspaces of $\mathbb R[x_1,\dots,x_n]_{4d}^{S,e}$ given by forms
\begin{eqnarray}
 p &:=& \sum_{i=1}^{m}\alpha_if_i(x) + \beta M_2^{2d} + \gamma M_{2d}^2 + \delta M_{2d}M_2^d \label{equ:maxdim}
\end{eqnarray}
where $\alpha_i ,\beta, \gamma, \delta \in \R^*$ and $f_i(x) \in B$ for $1\leq i\leq m$. 
In particular, we are interested in the constant
\begin{eqnarray*}
 M_{n,4d}^{(k)} \ := \ \max\{m & : & p\geq 0\Leftrightarrow p\geq 0\,\, \textrm{at all}\,\, k\textrm{-points}, \\ 
  & & \textrm{for all}\,\,p\,\,\textrm{with}\,\,\{f_{1},\dots,f_{m}\}\subseteq B\,\,\textrm{and}\,\,\alpha_i, \beta,\gamma,\delta\in\R^*\}.
\end{eqnarray*}

To the best of our knowledge nothing is known about these numbers so far. Note that $M_{n,4d}^{(k)}$ can be interpreted as a measure for the maximum dimension $m + 3$ such that at all $(m + 3)$-subspaces of forms given as in \eqref{equ:maxdim} (for arbitrary $\{f_{1},\dots,f_{m}\}\subseteq B\}$) nonnegativity can be decided at all $k$-points. Furthermore, we are interested in the set
\begin{eqnarray*}
A_{n,4d}^{(k)} &:=& \{p \in \mathbb R[x_1,\dots,x_n]_{4d}^{S,e}\ : \ p\geq 0\Leftrightarrow p\geq 0\,\, \textrm{at all}\,\, k\textrm{-points}\},
\end{eqnarray*}
i.e., the set of all forms in $\mathbb R[x_1,\dots,x_n]_{4d}^{S,e}$ for which nonnegativity can be decided at all $k$-points. Less is known about geometrical and topological properties of these sets. For example, a priori it is unclear whether these sets are connected or even convex. 

We summarize our upcoming results (Theorems \ref{Thm:Main}, \ref{Thm:Main2}, \ref{thm:maximum}, Corollary \ref{cor:convex}) in the following theorem.

\begin{thm}
\label{thm:mainmain}
Let $p := \sum_{i=1}^{m}\alpha_if_i(x) + \beta M_2^{2d} + \gamma M_{2d}^2 + \delta M_{2d}M_2^d$ with $\alpha_i ,\beta, \gamma, \delta \in \R^*$ and $n \geq 3$. Furthermore, let $B$ be the basis of $\mathbb R[x_1,\dots,x_n]_{4d}^{S,e}$ given by the power sum polynomials.
\begin{enumerate}
\item For $m + 2 \leq n$ and $p$ satisfying some extra conditions (see \eqref{Equ:ClassConditions2}) the set of $(m + 1)$-points is a test set for $p$,
\item For $4d \geq 12$, $n \in \{d - 1, d\}$ and $f_i \in B$ for $1 \leq i\leq m$ we have $M_{n,4d}^{(2)} = 1$,
\item For $4d \geq 12$, $n \in \{d - 1, d\}$ and $f_i \in B$ for $1 \leq i\leq m$ the set $A_{n,4d}^{(2)}$ is not convex.
\end{enumerate}

\end{thm}

\section{Subspaces of Even Symmetric Forms of Dimension Four}
\label{Sec:DimensionFour}

We start with the study of some $4$-dimensional subspaces. The main result in this section is the following theorem.
\begin{thm}
Let $n \geq 3$. The set of $2$-points is a test set for real even symmetric forms of the form 
\begin{eqnarray}
	p & := & \alpha M_{j_1}^{k_1} \cdots M_{j_r}^{k_r} + \beta M_2^{2d} + \gamma M_{2d}^2 + \delta M_{2d}M_2^d, \label{Equ:OurClass}
\end{eqnarray}
where $\alpha,\beta,\gamma,\delta \in \R^*$ and the following conditions are satisfied
\begin{eqnarray}
	& & j_1,\ldots,j_r \in 2\N, \ k_1,\ldots,k_r \in \N, \ \sum_{i = 1}^r j_i k_i = 4d, j_1 \notin \{2,2d\}, \label{Equ:ClassConditions} \\
	& & \text{and either } j_1,\ldots,j_r \leq 2d \text{ or } j_2,\ldots,j_r \in \{2,2d\}. \nonumber
\end{eqnarray}
\label{Thm:Main}
\end{thm}
 Hence, we have the following corollary.

\begin{cor}
 Let $p$ be of the form \eqref{Equ:OurClass} satisfying \eqref{Equ:ClassConditions}. Then nonnegativity of $p$ can be reduced to a finite number of semidefinite feasibility problems.
\end{cor}

For an introduction to semidefinite programming see, e.g., \cite{Blekherman:Parrilo:Thomas, Laurent:Survey}. Note that, in particular, semidefinite programs can be solved in time polynomial up to an additive $\varepsilon$-error.

\begin{proof}
 By Theorem \ref{Thm:Main}, $p$ is nonnegative if and only if it is nonnegative at all $2$-points. Hence, $p$ is nonnegative if and only if a finite number of binary forms are nonnegative. By Hilbert's theorem this is the case if and only if these binary forms are sums of squares, which can be decided by semidefinite programs (see \cite{Lasserre}). 
\end{proof}

Note that the special case of $(n, 2d) = (3,8)$ in the main Theorem \ref{Thm:Main} is considered by Harris in \cite{Harris}.
In order to prove this theorem we need some further results that follow a similar line as the results in \cite{Harris}. For given $p$ of the form \eqref{Equ:OurClass} satisfying \eqref{Equ:ClassConditions} let
$J(y)$ be the Jacobian of $\{M_{j_1}^{k_1} \cdots M_{j_r}^{k_r}, M_2^{2d},M_{2d}^2, M_{2d}M_2^d\}$ at the point $y$, i.e.,
$$J(y) : \R^4 \to \R^n, \ (\alpha, \beta, \gamma, \delta) \mapsto \left(\frac{\partial p}{\partial x_1}(y), \dots, \frac{\partial p}{\partial x_n}(y)\right)^T.$$ It is a $(n \times 4)$-matrix.

\begin{lemma}
\label{Lem:Jacobi}
 The following does hold for $y \in \R^n_{\geq 0}$: $\rank J(y) < 3$ if and only if $y$ is a $k$-point with $k\leq 2$.
\end{lemma}

\begin{proof}
 First, we prove the lemma for the case that $r = 1$, i.e., the first column of $J$ is given by the partial derivatives of $M_{j_1}^{k_1}$ with $j_1k_1 = 4d$, $j_1 \in 2\N$ and $j_1 \notin \{2,2d\}$ (see \eqref{Equ:ClassConditions}). Thus, the Jacobian $J$ is given by the following matrix
\begin{eqnarray*}
	J =   \left[ \begin {array}{cccc} 
k_1j_1 M_{j_1}^{k_1 - 1} x_1^{j_1-1} & 4\,dx_1M_2^{2d-1} & 4\,dx_1^{2d-1}M_{{2d}} & 2\,dx_1M_2^{d-1} \left( M_{{2d
}}+x_1^{2d-2}M_{{2}} \right) \\ 
 \vdots & \vdots & \vdots & \vdots \\
\noalign{\medskip} k_1j_1 M_{j_1}^{k_1 - 1} x_n^{j_1-1} & 4\,d x_n M_2^{2d-1} & 4\,dx_n^{2d-1}M_{{2d}} & 2\,d x_nM_2^{d-1} \left( M_{{2d}}+x_n^{2d-2}M_{{2}} \right) 
\end {array} \right]. 
\end{eqnarray*}

We investigate all $(3\times 3)$-minors of $J$. Due to the symmetry of $p$ (and therefore also $J$) in the variables $x_1,\ldots,x_n$ it suffices to restrict to $x_1,x_2,x_3$. Note that every $(3 \times 3)$-minor containing the fourth column of $J$ is irrelevant, since the fourth column is in the span of the second and the third column. Hence, if there exists a nonzero $(3 \times 3)$-minor containing the fourth column, then there also exists a nonzero $(3 \times 3)$-minor containing the first three columns. Thus, it only remains to investigate the leading principal $(3 \times 3)$-minor of $J$, which is due to calculation rules of determinants given by

$$ (4d)^3M_{j_1}^{k_1 - 1}M_{2d}M_{2}^{2d-1} q(x_1,x_2,x_3)$$ with

\begin{eqnarray}
	q(x_1,x_2,x_3) & :=	& 
	\det \left[ \begin {array}{ccc}
		x_1^{j_1-1} & x_1 & x_1^{2d-1} \\
		x_2^{j_1-1} & x_2 & x_2^{2d-1} \\
		x_3^{j_1-1} & x_3 & x_3^{2d-1} \\
	\end {array} \right]. \label{equ:matrix}
\end{eqnarray}
Note that $q$ does not equal the zero polynomial, since $j_1 \notin \{2,2d\}$ by assumption. Obviously, $q(x_1,x_2,x_3)$ vanishes if one entry is zero and, by \eqref{Equ:Determinant}, \eqref{Equ:Vandermonde} and \eqref{Equ:Schur}, we have 
\begin{eqnarray*}
	q(x_1,x_2,x_3) & =	& \Delta_3 \cdot (\pm 1) \cdot S_{(d_1,d_2,d_3)},
\end{eqnarray*}
with $(d_1,d_2,d_3) = (j_1-3,2d-2,1)$ for $j_1 > 2d$ and $(d_1,d_2,d_3) = (2d - 3,j_1 - 2,1)$ for $j_1 < 2d$. Since $\Delta_3 = (x_1-x_2)(x_1-x_3)(x_2-x_3)$, $q(x_1,x_2,x_3)$ vanishes if two entries are equal or if any entry is zero (since in this case the matrix in \eqref{equ:matrix} is singular). By Proposition \ref{Prop:SchurPolynomial} $q(x_1,x_2,x_3)$ has no further zeros on $\R^{3}_{> 0}$, because $S_{(d_1,d_2,d_3)}$ is a sum of monomial symmetric functions with nonnegative coefficients (the Kostka-numbers) and therefore $S_{(d_1,d_2,d_3)}(y) > 0$ for every $y \in \R^{3}_{> 0}$.

Since finally $M_{j_1}^{k_1 - 1}$, $M_{2d}$ and $M_{2}^{2d-1}$ are sums of squares, the leading principal  $(3 \times 3)$-minor of $J$ does not vanish for a $3$-point $y \in \mathbb R^3_{> 0}$. Hence, the minor vanishes if and only if one of $\{y_1, y_2, y_3\}$ is zero or at least two of them are equal, which is exactly the case if and only if $(y_1,y_2,y_3)$ is a $2$-point.

But this already implies that the rank of $J$ is less than three if and only if $y = (y_1,\ldots,y_n)$ is a 2-point. Assume that $J(y)$ has rank three. Then there exists a  
non-vanishing $(3 \times 3)$-minor of $J(y)$ given by the first three columns and three rows $i_1,i_2$ and $i_3$. Hence, by the upper argumentation, we have $y_{i_1} > y_{i_2} > y_{i_3} > 0$, i.e., $y$ is not a $2$-point. On the other hand, assume that $J(y)$ has rank two. Then every $(3 \times 3)$-minor of $J(y)$ given by the first three columns and three arbitrary rows $i_1,i_2$ and $i_3$ vanishes, i.e., by the upper argumentation $(y_{i_1},y_{i_2},y_{i_3})$ is a $2$-point. And since $\{i_1,i_2,i_3\}$ is an arbitrary subset of cardinality three of $\{1,\ldots,n\}$, we can conclude that $y$ is a $2$-point in total.\\

Now, we step over to the general case. Here, the first column of $J$ is given by the partial derivatives of $M_{j_1}^{k_1} \cdots M_{j_r}^{k_r}$ satisfying \eqref{Equ:ClassConditions}, i.e., the first column is given by
\begin{eqnarray*}
	\left(\sum_{i = 1}^r k_i j_i x_1^{j_i - 1} M_{j_i}^{k_i-1} \prod_{l \in \{1,\ldots,r\} \setminus \{i\}}M_{j_l}^{k_l}, \ \ldots, \ \sum_{i = 1}^r k_i j_i x_n^{j_i - 1} M_{j_i}^{k_i-1} \prod_{l \in \{1,\ldots,r\} \setminus \{i\}}M_{j_l}^{k_l} \right)^T.
\end{eqnarray*}
With the same argument as in the case $r = 1$, it suffices to investigate the leading principal $(3 \times 3)$-minor. By the calculation rules of the determinant this minor is given by
\begin{eqnarray}
	(4d)^2 M_{2d}M_{2}^{2d-1} \left(\sum_{i = 1}^r k_i j_i q_i(x_1,x_2,x_3) M_{j_i}^{k_i-1} \prod_{l \in \{1,\ldots,r\} \setminus \{i\}}M_{j_l}^{k_l}\right), \label{Equ:PrincipalMinor}
\end{eqnarray}
where
\begin{eqnarray}
	q_i(x_1,x_2,x_3) & :=	& 
	\det \left[ \begin {array}{ccc}
		x_1^{j_i-1} & x_1 & x_1^{2d-1} \\
		x_2^{j_i-1} & x_2 & x_2^{2d-1} \\
		x_3^{j_i-1} & x_3 & x_3^{2d-1} \\
	\end {array} \right] \ = \ \Delta_3 \cdot (\pm 1) \cdot S_{(d_1,d_2,d_3)},  \label{Equ:SchurFactorization}
\end{eqnarray}
with $(d_1,d_2,d_3) = (j_i-3,2d-2,1)$ for $j_i > 2d$ and $(d_1,d_2,d_3) = (2d - 3,j_i - 2,1)$ for $j_i < 2d$. Since all $j_i$ are even numbers (see \eqref{Equ:ClassConditions}) all $M_{j_i}$ are sums of squares, which, due to symmetry in the variables, only vanish at the origin. Hence, the zero set of \eqref{Equ:PrincipalMinor} only depends on the $q_i$ polynomials. Note that $q_i$ is the zero polynomial if and only if $j_i \in \{2,2d\}$. With the same argument as in the case $r = 1$ we know furthermore that all $q_i \neq 0$ vanish at $(y_1,y_2,y_3) \in \R_{\geq 0}^3$ if and only if $(y_1,y_2,y_3)$ is a 2-point. Hence, we are done if we can show that there exists a $q_i \neq 0$ and all $q_i$ have the same signum. But this follows from the conditions \eqref{Equ:ClassConditions}. They guarantee that $q_1 \neq 0$ and either all other $q_i = 0$ (and thus the signum of $q_1$ does not matter) or all $j_i \leq 2d$, which implies that the number of column changes needed to transform the defining matrix of each $q_i$ to the standard form \eqref{Equ:Determinant} is equal for all $i$ and thus the signum of all $q_i$ coincides.

Thus, the principal $(3 \times 3)$-minor indeed vanishes if and only if $(y_1,y_2,y_3) \in \R_{\geq 0}^3$ is a 2-point and analogously as in the case $r = 1$ this implies that the rank of $J$ is less than three if and only if $(y_1,\ldots,y_n)$ is a 2-point.
\end{proof}

If $y$ is not a $2$-point Lemma \ref{Lem:Jacobi} says that the solution space of $J(y)\cdot v = 0$ where $v := (\alpha, \beta, \gamma, \delta)$ is $1$-dimensional and in fact is obviously
 spanned by the following form that is clearly singular at $y$:
\begin{eqnarray}
 T_{y}(x) &:=& (\overline{M_2}^dM_{2d}(x) - \overline{M_{2d}}M_2(x)^d)^2, \label{equ:kern}
\end{eqnarray}
where $\overline{M_r} := M_r(y)$.

As a next step we prove that any sum of $2k$-th powers on the unit sphere can be formed by a $2$-point. This generalizes Lemma 2.6 in \cite{Harris}, where this is shown to be true for $2k = 4$. However, the proof follows the same line.

\begin{lemma}
\label{lem:sphere}
 Let $x \in \mathbb R^n_+$ be such that $M_2(x) = 1$ and $M_{2k}(x) = r$. Then there exists a $2$-point $z = (a,\dots,a,b)\in\mathbb R^n_+$ such that 
$M_{2}(z) = 1$ and $M_{2k}(z) = r$.
\end{lemma}

\begin{proof}
 We first note that the inequality $\frac{1}{n^{k-1}} \leq M_{2k}(x) \leq 1$ is true since we are dealing with points $x \in \mathbb R^n_+$ such that $M_2(x) = 1$ and by the equivalence of norms.
Let
\begin{eqnarray*}
 z_{\alpha} &:=& \left(\frac{\cos\alpha}{\sqrt{n-1}},\dots,\frac{\cos\alpha}{\sqrt{n-1}},\sin\alpha\right).
\end{eqnarray*}
Then $f(\alpha) := M_{2k}(z_{\alpha}) = \frac{\cos^{2k}\alpha}{(n-1)^{k-1}} + \sin^{2k}\alpha$. In particular, 
$M_{2}(z_\alpha) = 1$ for all $\alpha$ as well as $f(\frac{\pi}{2}) = 1$ and, since $\cos(\arcsin(x)) = \sqrt{1 - x^2}$, it follows that
\begin{eqnarray*}
&&f\left(\arcsin\left(\frac{1}{\sqrt{n}}\right)\right)\ = \ \frac{(1-1/n)^k}{(n-1)^{k-1}} + \frac{1}{n^k}\ =\ \frac{1}{n^{k-1}}.
\end{eqnarray*}
Hence, by the intermediate value theorem for all $r$ with $\frac{1}{n^{k-1}} \leq r \leq 1$ there exists $\alpha^* \in [\arcsin(\frac{1}{\sqrt{n}}),\frac{\pi}{2}]$
 such that $f(\alpha^*) = r$ and $z_{\alpha^*} = (a,\dots,a,b)$.
\end{proof}

Now, we can prove our main theorem.

\begin{proof}(Theorem \ref{Thm:Main})
 We need to prove that if $\alpha M_{j_1}^{k_1} \cdots M_{j_r}^{k_r} + \beta M_2^{2d} + \gamma M_{2d}^2 + \delta M_{2d}M_2^d$ is nonnegative at all $2$-points, then it is also nonnegative globally. Suppose $p$ is nonnegative at all $2$-points but not nonnegative. Let $-\lambda :=\min_{x\in\mathbb S^{n-1}} p < 0$ denote the minimum value of $p$ over the
unit sphere and let $y = (y_1,\ldots,y_n) \in \mathbb S^{n-1}$ be the minimizer such that $p(y) = -\lambda$ (note that it suffices to restrict to the unit sphere due to homogeneity). Since the degree of every variable in every monomial of $p$ is even, we can assume w.l.o.g. that $y \in \mathbb S^{n-1}_+$.  Then $q(x) := p(x) + \lambda M_2^{2d}(x) \geq 0$
and $q(y) = 0$. By assumption $y$ is not a $k$-point with $k\leq 2$ (because $p$ is nonnegative at these points). By Lemma \ref{Lem:Jacobi} we have $\rank J(y) = 3$ and hence $q = k\cdot T_{y}(x) , k > 0$ with $T_y$ as in \eqref{equ:kern}, since $q$ is in the kernel of $J(y)$. Thus, $q(x) = 0$ whenever $M_{2d} = \overline{M_{2d}}$, i.e.,
$x_1^{2d}+\dots + x_n^{2d} = y_1^{2d}+\dots + y_n^{2d}$.
By Lemma  \ref{lem:sphere} there exists a $2$-point $z = (a,a,\dots,a,b)$ such that $(n-1)a^2 + b^2 = 1$ and $(n-1)a^{2d} + b^{2d} = \overline{M_{2d}}$. But this implies
$p(z) = -\lambda$ which is a contradiction since $p$ is nonnegative at all $2$-points.
\end{proof}

\subsection{An Exemplary Application}
In this subsection we briefly want to demonstrate how our Theorem \ref{Thm:Main} can be applied to test nonnegativity of an example class and even how to derive a computeralgebraically generated semialgebraic description of a certain subcone of the cone of nonnegative even symmetric forms. 

The key fact on an application side is that checking whether forms are nonnegative at $2$-points can be reduced to checking nonnegativity of univariate polynomials, which can be done efficiently by checking numerically (i.e., under usage of SDP-methods; see e.g. \cite{Laurent:Survey} for further details) whether these polynomials are sums of squares (due to Hilbert's theorem). Alternatively this can also be done by using quantifier elimination methods, which happen to work quite efficiently for univariate polynomials of sufficiently low degree.

Our first example shows that the same set of coefficients yields different results concerning nonnegativity when the number of variables increases.\\

\textbf{Example}: Consider the form
\begin{eqnarray*}
p(x_1,x_2,x_3) &:=& M_4^3 - \frac{1}{10}M_2^6 + M_6^2 + M_6M_2^3.
\end{eqnarray*}
By Theorem \ref{Thm:Main}, $p\geq 0$ if and only if the two binary forms $p(x_1,x_2,0)$ and $p(x_1,x_1,x_2)$ are nonnegative. By dehomogenizing the binary forms this is the case if and only if the following two univariate polynomials are nonnegative:

\begin{eqnarray}
 &&\frac{29}{10}\,{x}^{12}+\frac {12}{5}\,{x}^{10}+\frac{9}{2}\,{x}^{8}+2\,{x}^{6}+\frac{9}{2}\,{x}^{4}+\frac {12}{5}\,{x}^{2}+\frac {29}{10} \label{equ:univariat},\\
&&\frac{108}{5}\,{x}^{12}+\frac{24}{5}\,{x}^{10}-2\,{x}^{6}+12\,{x}^{4}+\frac {24}{5}\,{x}^{2}+\frac{29}{10}\nonumber.
\end{eqnarray}

Since these polynomials are obviously nonnegative, we conclude $p\geq 0$. However, consider now the same form in four variables, i.e., 
\begin{eqnarray*}
p(x_1,x_2,x_3,x_4) &:=& M_4^3 - \frac{1}{10}M_2^6 + M_6^2 + M_6M_2^3.
\end{eqnarray*}
By Theorem \ref{Thm:Main}, $p\geq 0$ if and only if the four binary forms $p(x_1,x_2,0,0)$, $p(x_1,x_1,x_2,0)$, $p(x_1,x_1,x_1,x_2)$ and $p(x_1,x_1,x_2,x_2)$ are nonnegative. By dehomogenizing, the first two binary forms are exactly the polynomials in \eqref{equ:univariat} from which we already know that they are nonnegative. Hence, $p$ is nonnegative if and only if the following two univariate polynomials are nonnegative:
\begin{eqnarray*}
 &&{\frac {441}{10}}\,{x}^{12}-{\frac {324}{5}}\,{x}^{10}-{\frac {135}{2}
}\,{x}^{8}-18\,{x}^{6}+{\frac {45}{2}}\,{x}^{4}+{\frac {36}{5}}\,{x}^{
2}+{\frac {29}{10}},\\
&&{\frac {108}{5}}\,{x}^{12}+{\frac {48}{5}}\,{x}^{10}-24\,{x}^{8}-88\,{
x}^{6}-24\,{x}^{4}+{\frac {48}{5}}\,{x}^{2}+{\frac {108}{5}}.
\end{eqnarray*}

It is easy to check that these polynomials are indefinite. Hence, $p$ is not a nonnegative form.\\

Now, we investigate the 4-variate dodecics given by
\begin{eqnarray}
p(x_1,x_2,x_3,x_4) & := & \alpha M_4^3 + \beta M_2^6 + \gamma M_6^2 + M_6M_2^3.\label{equ:bsp}
\end{eqnarray}

It turns out that quantifier elimination methods are not suitable to decide, for which $(\alpha, \beta, \gamma) \in \R^3$ the form $p$ is nonnegative, since the problem is too complex. Here, we used the quantifier elimination package \textsc{SyNRAC} for \textsc{Maple} (see \cite{Anai:Yanami}), which terminated without a solution after round about $18$ minutes.

Anyhow, application of Theorem \ref{Thm:Main} allows to quickly derive a computeralgebraical description of the desired semialgebraic set. We successively apply Theorem \ref{Thm:Main} on polynomials $p$ given by the parameter sets $\{(\alpha,\beta,\gamma) \ : \ \alpha \in \{1,2\}, (\beta,\gamma) \in [-10,10]^2 \cap \Z^2\}$ and $\{(\alpha,\beta,\gamma) \in [-4,4] \cap \Z^3\}$. The nonnegativity region of the corresponding polynomials in the parameter sets are depicted in the three pictures of Figure \ref{Fig:ParameterSets}.

\begin{figure}
\includegraphics[width=0.3\linewidth]{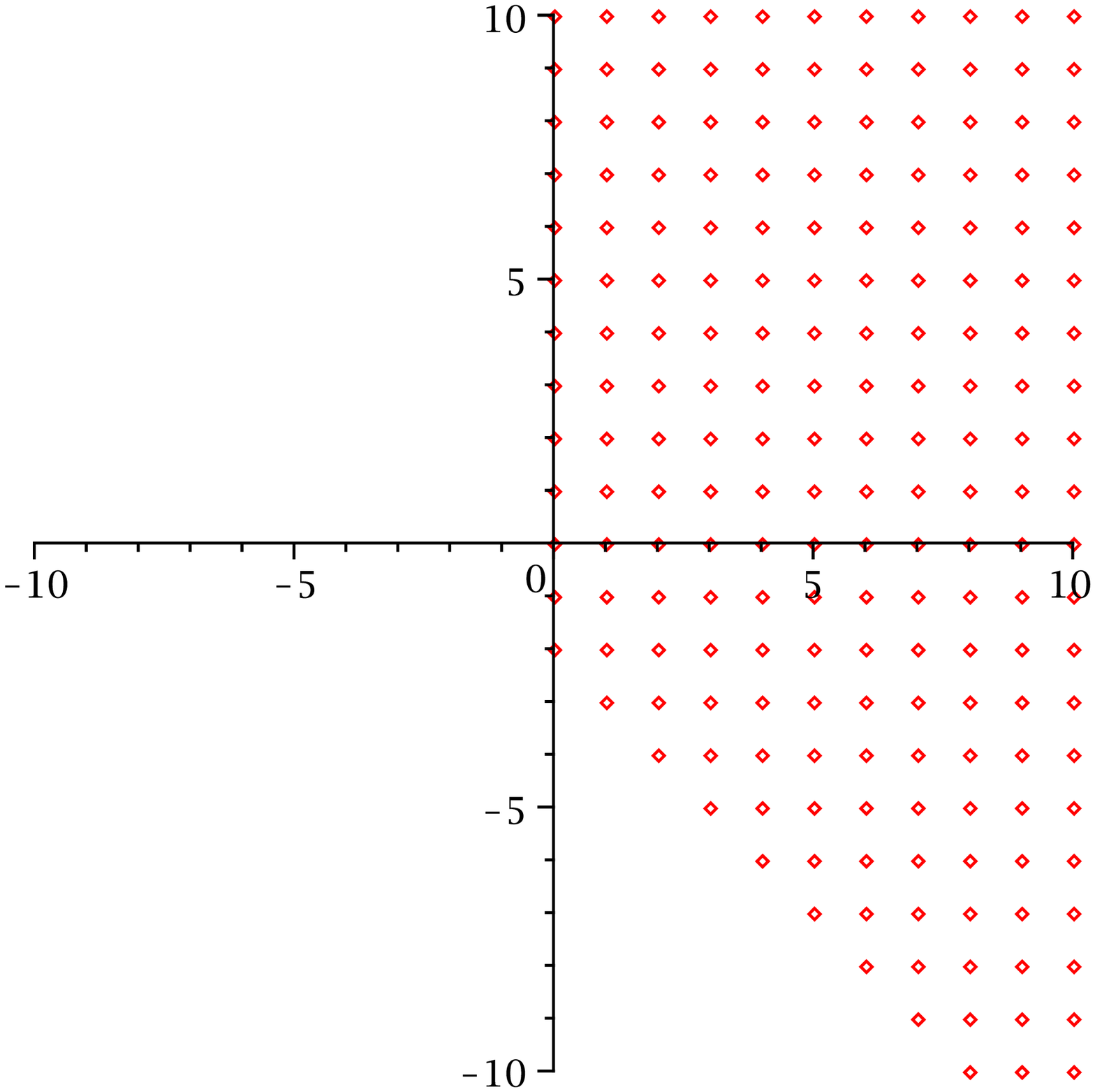}
\includegraphics[width=0.3\linewidth]{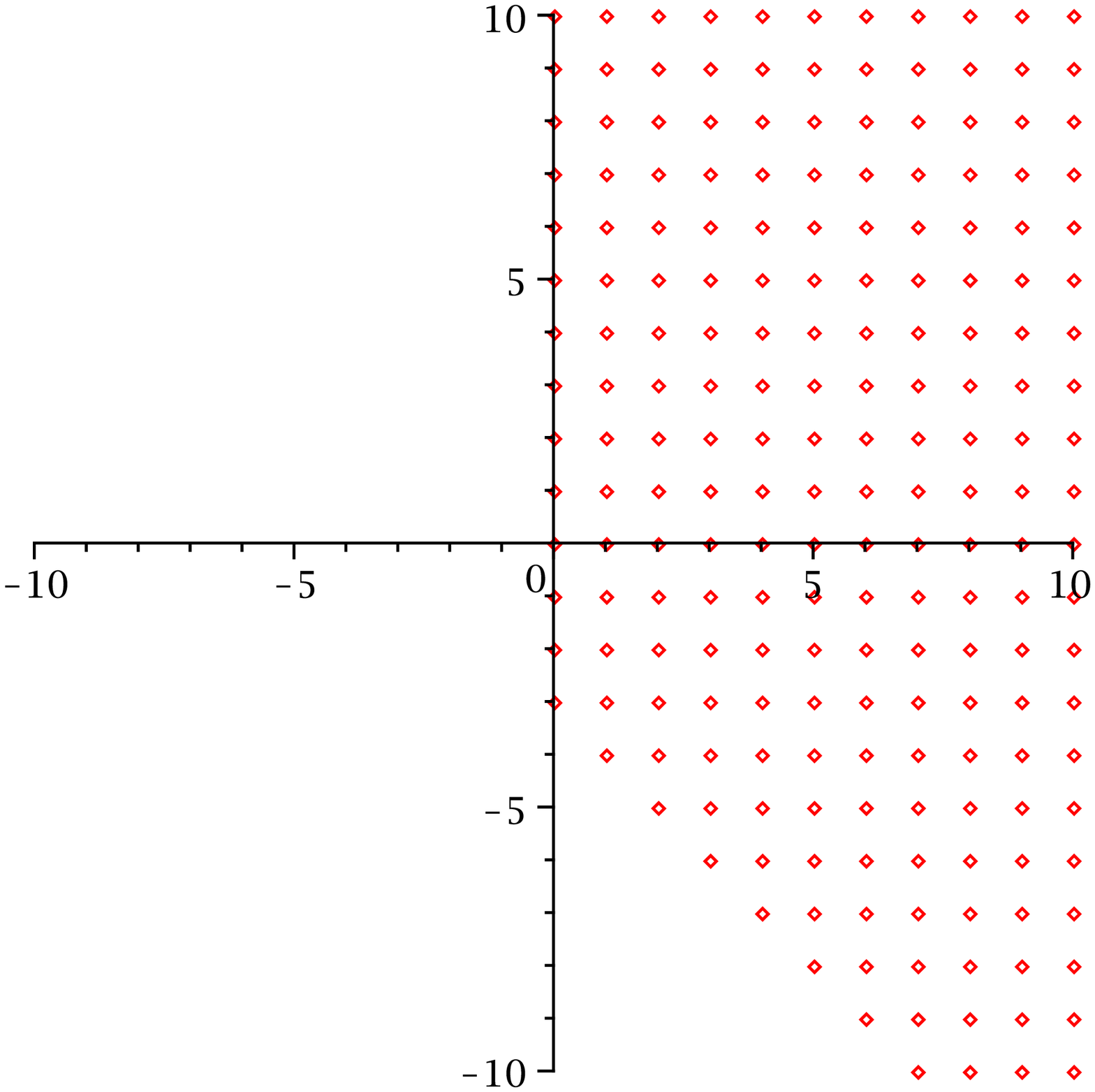}
\includegraphics[width=0.3\linewidth]{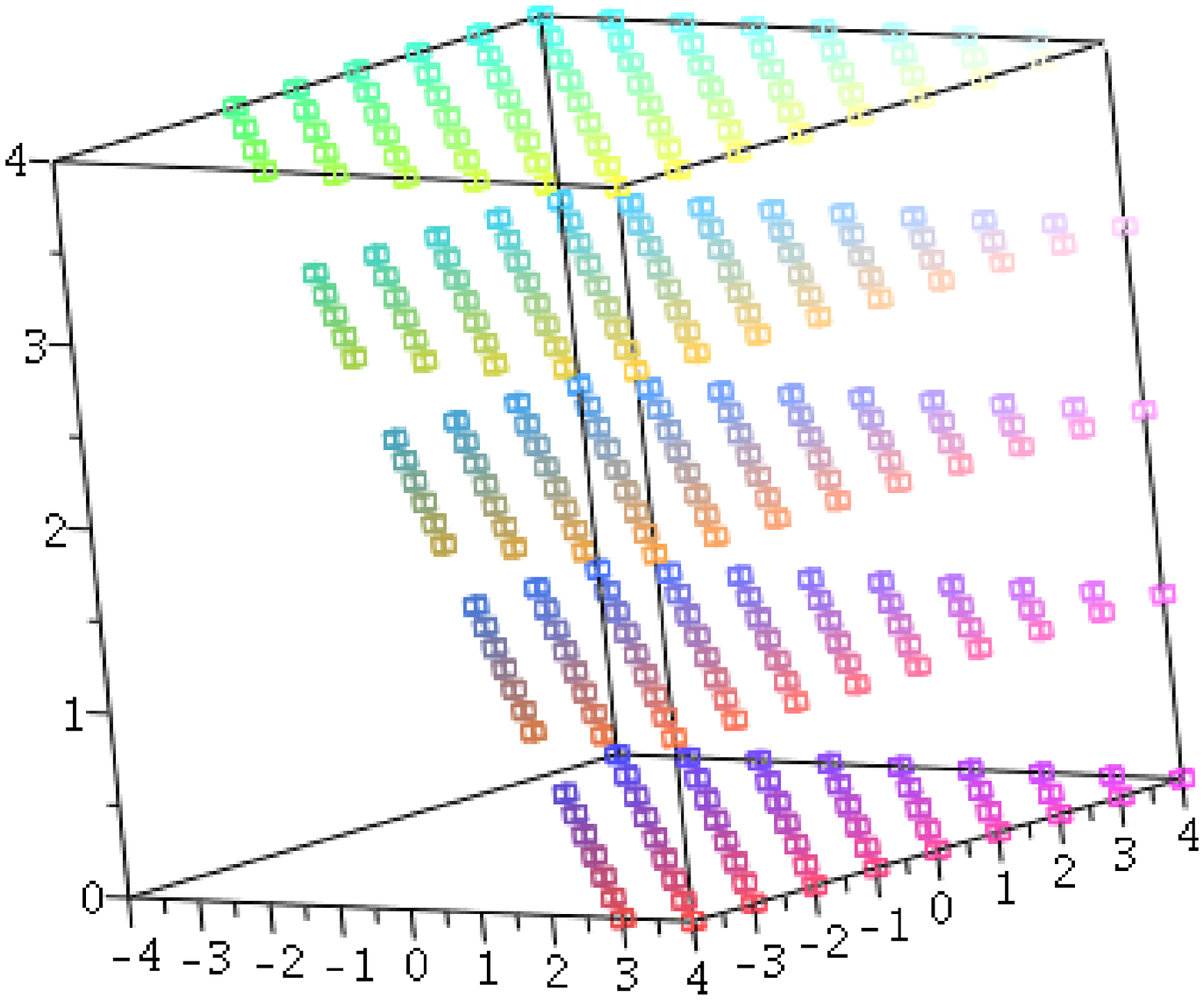}
\caption{The nonnegativity region of polynomials of the form \eqref{equ:bsp} in the parameter sets $\{(\alpha,\beta,\gamma) \ : \ \alpha \in \{1,2\}, (\beta,\gamma) \in [-10,10]^2 \cap \Z^2\}$ and $\{(\alpha,\beta,\gamma) \in [-4,4] \cap \Z^3\}$.}
\label{Fig:ParameterSets}
\end{figure}

The computed region of nonnegativity obviously is polyhedral. In fact, the approximated set of parameters, which yield nonnegative polynomials $p$, can easily be identified as 
$$\{(\alpha,\beta,\gamma) \in \R^3 \ : \ \beta \geq 0, \alpha + \beta + \gamma + 1 \geq 0\}.$$
We furthermore checked with \textsc{SOSTools} (see \cite{Parrilo:SOSTOOLS}) for various examples (e.g., $\alpha = 1$, $\beta = 0$, $\gamma \in \{-1,-2\}$) located on the boundary of the polyhedra described by the upper set, if the corresponding polynomials $p$ are sums of squares. Indeed, this was always the case. Hence, by convexity one would expect that every nonnegative form is a sum of squares in this particular subcone.

\section{Subspaces of Even Symmetric Forms of Arbitrary Dimension}
\label{Sec:arbitrary}
A natural question is in how far the constructions in Section \ref{Sec:DimensionFour} can be generalized to higher dimensional subspaces in the vector space of even symmetric forms of degree $4d$ in $n$ variables. We show that with some obvious modifications such generalizations are indeed possible. However, the price to pay is an adjustment of the number of variables to the ambient dimension of the investigated subspaces of forms of degree $4d$.\\

Before we can introduce the formal setting for this section, we need to give one more definition. Let $V \subset (\N^*)^k$. For every vector $(v_1,\ldots,v_k) \in V$, which is not a $(k-1)$-point, we denote by $\sig_v$ the permutation, which maps $v$ to the unique vector $\sig_v(v)$ with $\sig_v(v_1) > \cdots > \sig_v(v_k)$. For every $v,w \in V$ with $v,w$ not being $(k-1)$-points, we say that $v$ and $w$ are \textit{identically oriented ordered}, if $\sign(\sig_v) \cdot \sign(\sig_w) = 1$. We say that $V$ is \textit{identically oriented ordered}, if every pair $v,w \in V$ with $v,w$ not being $(k-1)$-points, is identically oriented ordered.\\

We consider the following class of even symmetric forms: Let for $m \leq n$

\begin{eqnarray}
	p(x_1,\dots,x_n) & := & \sum_{i=1}^{m-2}\alpha_if_i(x) + \beta M_2^{2d} + \gamma M_{2d}^2 + \delta M_{2d}M_2^d, \label{Equ:OurClass2}
\end{eqnarray}

where $\alpha_1,\ldots,\alpha_m,\beta,\gamma,\delta \in \R^*$ and $f_i(x) = M_{j_{(i,1)}}^{k_{(i,1)}} \cdots M_{j_{(i,r_i)}}^{k_{(i,r_i)}}$ with $r_i \in \N^*$ is a product of power sums $M_j$ of degree $4d$ such that the following conditions (which are a natural generalization of \eqref{Equ:ClassConditions}) hold:
\begin{eqnarray}
	& & j_{(i,1)},\ldots,j_{(i,r_i)} \in 2\N, \ k_{(i,1)},\ldots,k_{(i,r_i)} \in \N, \ \sum_{l = 1}^{r_i} j_{(i,l)} k_{(i,l)} = 4d, \nonumber \\ 
	& & j_{(i,1)} \notin \{2,2d\} \cup \bigcup_{l = 1}^{i-1} \{j_{(l,1)},\ldots,j_{(l,r_l)}\} \text{ for every } 1 \leq i \leq m-2, \text{ and the set } \label{Equ:ClassConditions2} \\
	& & \varPsi \ := \ \{(j_{(1,l_1)}, j_{(2,l_2)}, \ldots,j_{(m-2,l_{m-2})},2,2d) \ : \ 1 \leq l_i \leq r_i \text{ for every } 1 \leq i \leq m-2\} \nonumber \\
	& & \text{is identically oriented ordered.} \nonumber
\end{eqnarray}

Hence, $p$ is an even symmetric form in $n$ variables of degree $4d$ in an $(m+1)$-dimensional subspace of even symmetric forms of degree $4d$. Again, as in Section \ref{Sec:DimensionFour}, we denote the Jacobian of $p$ at the point $y$ by $J(y)$, which is an $n\times (m+1)$-matrix. Note that for $m = 3$ the fact that $\varPsi$ is identically oriented ordered is equivalent to the conditions \eqref{Equ:ClassConditions}. The extension of the conditions \eqref{Equ:ClassConditions2} w.r.t. the conditions \eqref{Equ:ClassConditions} become necessary for a generalization to arbitrary dimensions of the subspace for two reasons: Firstly, we need to guarantee that specific $(m \times m)$-minors of interest in $J(y)$ do not equal the zero polynomial. Recall that we similarly had to guarantee that the investigated leading principal $(3 \times 3)$-minor in the proof of Lemma \ref{Lem:Jacobi} did not equal the zero polynomial. Secondly, in the case that our investigated minor can (by calculation rules of the determinant) be rewritten as a sum of simpler determinants, we need to guarantee that all these determinants have the same signum. Recall that we also already had to do this in the $4$-dimensional case when the first summand was a product of different power sums (see proof of Lemma \ref{Lem:Jacobi}).

\begin{lemma}
\label{Lem:Jacobi2}
 The following does hold for $y \in \R^n_{\geq 0}$: $\rank J(y) < m$ if and only if $y$ is a $k$-point with $k\leq m-1$.
\end{lemma}

\begin{proof}
Basically, the proof works analogously to the one in Lemma \ref{Lem:Jacobi} up to the fact that we investigate $(m \times m)$-minors instead of $(3 \times 3)$-minors.

The last three columns of $J$ agree with those in the dimension four case (see proof of Lemma \ref{Lem:Jacobi}). For $1 \leq i \leq m-2$ the $i$-th column  of $J$ is given by
\begin{eqnarray*}
	\left(\begin{array}{c}
		\sum_{l = 1}^{r_i} k_{(i,l)} j_{(i,l)} x_1^{j_{(i,l)}-1} M_{j_{(i,l)}}^{k_{(i,l)}-1} \prod_{s \in \{1,\ldots,r_i\} \setminus \{l\}}M_{j_{(i,s)}}^{k_{(i,s)}} \\
		\vdots \\
		\sum_{l = 1}^{r_i} k_{(i,l)} j_{(i,l)} x_n^{j_{(i,l)}-1} M_{j_{(i,l)}}^{k_{(i,l)}-1} \prod_{s \in \{1,\ldots,r_i\} \setminus \{l\}}M_{j_{(i,s)}}^{k_{(i,s)}} \\
	\end{array}\right).
\end{eqnarray*}

Our goal is to find an $(m \times m)$-minor, which vanishes only on $k$-points with $k \leq m-1$. With the same arguments on the last column of $J$ and the symmetry of the variables, we can restrict to the leading principal $(m \times m)$-minor of $J$, as in the proof of Lemma \ref{Lem:Jacobi}. By calculation rules of the determinant this minor is given by $(4d)^2 M_{2d}M_{2}^{2d-1}$ times

\begin{eqnarray}
	\sum_{1 \leq l_1 \leq r_1, \cdots, 1 \leq l_m \leq r_m} q_{(l_1,\ldots,l_{m-2})}(x_1,\ldots,x_m) \cdot \sum_{i = 1}^{m-2} k_{(i,l_i)} j_{(i,l_i)}  M_{j_{(i,l_i)}}^{k_{(i,l_i)}-1} \prod_{s \in \{1,\ldots,r_i\} \setminus \{l_i\}}M_{j_{(i,s)}}^{k_{(i,s)}}, \label{Equ:PrincipalMinor2}
\end{eqnarray}
where
\begin{eqnarray}
	&   & q_{(l_1,\ldots,l_{m-2})}(x_1,\ldots,x_m) \nonumber \\
	& := & \det \left[ \begin {array}{ccccc}
		x_1^{j_{(1,l_1)}-1} & \cdots & x_1^{j_{(m-2,l_{m-2})}-1} & x_1 & x_1^{2d-1} \\
		\vdots & \ddots & \vdots & \vdots & \vdots \\
		x_m^{j_{(1,l_1)}-1} & \cdots & x_m^{j_{(m-2,l_{m-2})}-1} & x_m & x_m^{2d-1} \\
	\end{array} \right] \label{Equ:SchurFactorization2} \\
	& = & \Delta_m \cdot (\pm 1) \cdot S_{(d_1,\ldots,d_m)}, \nonumber
\end{eqnarray}
for appropriate choices of $d_i$, which we discuss in detail later. First, notice that all power sums involved in \eqref{Equ:PrincipalMinor2} are sums of squares since all $j_{(i,l_i)}$ are even by condition \eqref{Equ:ClassConditions2}. Hence the zero set of \eqref{Equ:PrincipalMinor2} only depends on the $q_{(l_1,\ldots,l_{m-2})}$ polynomials as in the dimension four case. Note that $q_{(l_1,\ldots,l_{m-2})}$ is the zero polynomial if and only if two columns of the matrix \eqref{Equ:SchurFactorization2} coincide, which is the case precisely if and only if $(j_{(1,l_1)},\ldots,j_{(m-2,l_{m-2})},2,2d)$ is an $(m-1)$-point. In particular, the minor \eqref{Equ:PrincipalMinor2} is not the zero polynomial, since the condition $j_{(i,1)} \notin \{2,2d\} \cup \bigcup_{l = 1}^{i-1} \{j_{(l,1)},\ldots,j_{(l,r_l)}\}$ guarantees that at least $q_{(1,\ldots,1)}$ is not the zero polynomial.

Notice that for all nonzero polynomials $q_{(l_1,\ldots,l_{m-2})}$ the factor $\pm 1$ is given by 
$$\sign(\sig_{(j_{(1,l_1)},\ldots,j_{(m-2,l_{m-2})},2,2d)}(j_{(1,l_1)},\ldots,j_{(m-2,l_{m-2})},2,2d)),$$
and each $d_i$ equals the $i$-the entry of the image vector of this permutation minus $(m-i)$ (see \eqref{Equ:Determinant}). Since we assumed in \eqref{Equ:ClassConditions2} that $\varPsi$ is identically oriented ordered, we know in particular that all signa of permutations corresponding to $(j_{(1,l_1)},\ldots,j_{(m-2,l_{m-2})},2,2d)$ coincide. Thus, we are done if we can show that every nonzero $q_{(l_1,\ldots,l_{m-2})}$ vanishes exactly at all $(m-1)$-points. But this is obviously the case since $\Delta_m$ vanishes if and only if two entries $x_i$ and $x_j$ coincide, and the whole matrix given in \eqref{Equ:SchurFactorization2} vanishes for $x_j = 0$ since it has a zero-column in this case. By Proposition \ref{Prop:SchurPolynomial} we can write $S_{(d_1,\ldots,d_m)}$ as a sum of monomial symmetric functions times a nonnegative Kostka-number, which guarantees that $q_{(l_1,\ldots,l_{m-2})}$ does not vanish on a non-$(m-1)$-point in the strict positive orthant. The rest of the argumentation is analogously to the proof of Lemma \ref{Lem:Jacobi}.
\end{proof}

With this lemma, we can prove an analogous version of Theorem \ref{Thm:Main}

\begin{thm}
Let $m \leq n$. The set of $(m-1)$-points is a test set for all even symmetric forms of the form $p := \sum_{i=1}^{m-2}\alpha_if_i(x) + \beta M_2^{2d} + \gamma M_{2d}^2 + \delta M_{2d}M_2^d$ as in \eqref{Equ:OurClass2} such that the conditions \eqref{Equ:ClassConditions2} are satisfied.
\label{Thm:Main2}
\end{thm}

\begin{proof}
 We need to prove that if $p = \sum_{i=1}^{m-2}\alpha_if_i(x) + \beta M_2^{2d} + \gamma M_{2d}^2 + \delta M_{2d}M_2^d$ is nonnegative at all $(m-1)$-points, it is also nonnegative globally. Suppose this is not the case. Let $-\lambda :=\min_{x\in\mathbb S^{n-1}} p < 0$ denote the minimum value of $p$ over the
unit sphere and let $y = (y_1,\dots,y_n) \in \mathbb S^{n-1}$ be the minimizer such that $p(y) = -\lambda$. Since the degree of every variable in every monomial of $p$
 is even, we can assume w.l.o.g. that $y \in \mathbb S^{n-1}_+$.  Then $q(x) := p(x) + \lambda M_2^{2d}(x) \geq 0$
and $q(y) = 0$. Since, by assumption, $y$ is not a $k$-point with $k\leq (m-1)$ (because $p$ is nonnegative at these points), $y$ must have at least $m$ distinct entries.
By Lemma \ref{Lem:Jacobi2} we have $\rank J(y) = m$ and hence $q = k\cdot T_{y}(x)$ with $T_y$ as in \eqref{equ:kern} and $k > 0$. Thus $q(x) = 0$ whenever 
$x_1^{2d}+\dots + x_n^{2d} = y_1^{2d}+\dots + y_n^{2d}$.
By Lemma  \ref{lem:sphere} there exists a $2$-point $z = (a,\ldots,a,b)$ such that $(n-1)a^2 + b^2 = 1$ and $(n-1)a^{2d} + b^{2d} = \overline{M_{2d}}$. But this implies
$p(z) = -\lambda$ which is a contradiction since $p$ is nonnegative at all $2$-points.
\end{proof}

\textbf{Example}: Consider even symmetric forms in $n = 6$ variables of degree $4d = 32$. By Timofte's theorem these forms are nonnegative if and only they are nonnegative at all $8$-points, which obviously is a useless information in this case. However, considering appropriate subspaces of dimension $m + 1 \leq 7$, Theorem \ref{Thm:Main2} states that nonnegativity at these subspaces can be checked at $(m - 1)$-points. 

\section{$k$-point Certificates at Maximal Subspaces}
\label{Sec:konstante}
We have seen that the number of components to check for nonnegativity of even symmetric forms can be reduced by considering appropriate subspaces containing
the three power sums $$M_2^{2d}, M_{2d}^2, M_{2d}M_2^d.$$
Recall that $\mathbb R[x_1,\dots,x_n]_{4d}^{S,e}$ is the vector space of \emph{even} symmetric forms in $n$ variables of degree $4d$ and $B$ be the basis given by the power sum polynomials. In this section we analyze the problem to determine for fixed $k\in \mathbb N$ the maximum dimension of all subspaces of forms given as
\begin{eqnarray*}
 p &:=& \sum_{i=1}^{m}\alpha_if_i(x) + \beta M_2^{2d} + \gamma M_{2d}^2 + \delta M_{2d}M_2^d
\end{eqnarray*}
where $f_i(x) \in B$ for $1\leq i\leq m$ and where nonnegativity can be checked at all $k$-points. Recall that
\begin{eqnarray*}
 M_{n,4d}^{(k)} \ = \ \max\{m & : & p\geq 0\Leftrightarrow p\geq 0\,\, \textrm{at all}\,\, k\textrm{-points}, \\ 
  & & \textrm{for all}\,\,p\,\,\textrm{with}\,\,\{f_{1},\dots,f_{m}\}\subseteq B\,\,\textrm{and}\,\,\alpha_i, \beta,\gamma,\delta\in\R^*\}.
\end{eqnarray*}
Note that $M_{s,4d}^{(k)} = M_{L(n,4d),4d}^{(k)}$ for $s > L(n,4d)$ where $L(n,4d) := \dim \R[x_1,\dots,x_n]_{4d}^{S,e}$.
As an illustrative example consider the quantity $M_{3,12}^{(2)}$. We have  $\dim \mathbb R[x_1,x_2,x_3]_{12}^{S,e} = 7$. An element $p\in \mathbb R[x_1,x_2,x_3]_{12}^{S,e}$ can be represented as
\begin{eqnarray*}
 p &=& \alpha_1M_6M_4M_2 + \alpha_2M_4^3 + \alpha_3M_4^2M_2^2 + \alpha_4M_4M_2^4 + \beta M_2^{6} + \gamma M_{6}^2 + \delta M_{6}M_2^3 .
\end{eqnarray*}
Fixing the last three terms, the question is, how many of the first four terms can be used in the representation of $p$ to decide nonnegativity of $p$ via nonnegativity at all $2$-points for any choice of the remaining four power sums. For example, if $M_{3,12}^{(2)} = 1$, then the forms 
\begin{eqnarray*}
 p &=& \beta M_2^{6} + \gamma M_{6}^2 + \delta M_{6}M_2^3 + \alpha_1q \text{ with } q\in \{M_6M_4M_2, M_4^3, M_4^2M_2^2, M_4M_2^4\}
\end{eqnarray*}
 would be nonnegative if and only if they are nonnegative at all $2$-points, and there exist $f_1, f_2 \in \{M_6M_4M_2, M_4^3, M_4^2M_2^2, M_4M_2^4\}$ such that
\begin{eqnarray*}
 p &=& \beta M_2^{6} + \gamma M_{6}^2 + \delta M_{6}M_2^3 + \alpha_1f_1 + \alpha_2f_2
\end{eqnarray*}
 is nonnegative at all $2$-points but not globally nonnegative.
 In fact, we prove that $M_{3,12}^{(2)} = 1$ by a much stronger result, which partially follows from Theorem \ref{Thm:Main}. The next Lemma is a generalization of Lemma 3.3 in \cite{Harris}.

\begin{lemma}
\label{Lem:Jacobi3}
 Let $d \geq 3$, $y \in \mathbb R^n_{\geq 0}$ and $\varphi : \mathbb R^n \to \mathbb R^3$ defined by
$$\varphi : (x_1,\dots,x_n) \mapsto (M_2, M_{2d - 2}, M_{2d}).$$
Then $y \in \partial \varphi(\mathbb R^3)$ if and only if $y$ is a $2$-point.
\end{lemma}

\begin{proof}
 The Jacobian $\Jac(y)$ of $\varphi$ at a point $y$ is a ($3\times n$)-matrix. Then $y \in \partial \varphi(\mathbb R^3)$ if and only if $\rank\Jac(y) < 3$. By symmetry, it suffices to investigate the leading principal $(3\times 3)$-minor corresponding to the first three rows and columns. By \eqref{Equ:Determinant}, \eqref{Equ:Vandermonde} and \eqref{Equ:Schur} this minor is given by $2\cdot (2d - 2)\cdot 2d\cdot \Delta_3\cdot S_{2d - 2, 2d - 3, 2}$.  As in the proof of Lemma \ref{Lem:Jacobi} this minor vanishes if and only if $y$ is a $2$-point.
\end{proof}

Note that we have $M_{n,4}^{(2)} = 0$ and $M_{n,8}^{(2)} = 2$ by Timofte's theorem.

\begin{thm}
\label{thm:maximum}
  Let $p := \sum_{i=1}^{m}\alpha_if_i(x) + \beta M_2^{2d} + \gamma M_{2d}^2 + \delta M_{2d}M_2^d \in \mathbb R[x_1,\dots,x_n]_{4d}^{S,e}$ with $\alpha_1, \dots, \alpha_m, \beta, \gamma, \delta \in \R^*$. Then for $4d \geq 12$ we have
$$M_{n,4d}^{(2)} \ = \ 1\quad \textrm{for}\quad n \in \{d - 1, d\}.$$
\end{thm}

\begin{proof}
  Let $4d \geq 12$. Furthermore, since $n \in \{d - 1, d\}$ the additional power sums $f_j(x) = M_{j_1}^{k_1} \cdots M_{j_r}^{k_r} \in B$ have the property that $j_k \leq 2d$ for $1\leq k\leq r$. This is because every even symmetric form in $n$ variables can uniquely be represented in the first $n$ power sums of even power (see Section \ref{Sec:Preliminaries}). Hence, by Theorem \ref{Thm:Main} we have $M_{n,4d}^{(2)}\geq 1$ and it remains to show that there is a choice of two power sums  $f_1, f_2 \in B$ such that
\begin{eqnarray*}
 p &=& \beta M_2^{6} + \gamma M_{6}^2 + \delta M_{6}M_2^3 + \alpha_1f_1 + \alpha_2f_2
\end{eqnarray*}
is nonnegative at all $2$-points but not nonnegative globally. For this we generalize the construction in \cite{Harris} where the author proves this for even symmetric ternary forms of degree $12$. For $y \in \mathbb R^n$ define
\begin{eqnarray*}
 p_y(x_1,\dots,x_n) &:=& (\overline{M_2}^dM_{2d} - \overline{M_{2d}}M_2^d)^2 +  (\overline{M_2}^dM_{2d-2}M_2 - \overline{M_{2d-2}}\,\overline{M_2}M_2^d)^2.
\end{eqnarray*}
The form $p$ is precisely of our desired form $p = \beta M_2^{2d} + \gamma M_{2d}^2 + \delta M_{2d}M_2^d + \alpha_1f_1 +\alpha_2f_2$ with $f_1 = M_{2d-2}^2M_2^2$ and 
$f_2 = M_{2d-2}M_2M_2^d$. Note that $f_1, f_2 \in B$ if $n\in \{d - 1, d\}.$
By construction, for $x\in \mathbb S^{n-1}$ we have $M_2(x) = 1$ and hence $p_y(x) = 0$ if and only if $M_{2d} = \overline{M_{2d}}$ and $M_{2d - 2} = \overline{M_{2d - 2}}$.
Note that $\frac{1}{n^{d-2}}\leq M_{2d - 2}\leq 1$ and $\frac{1}{n^{d-1}}\leq M_{2d}\leq 1$ (see proof of Lemma \ref{lem:sphere}). Fix $\Theta \in (\frac{1}{n^{d-2}},1)$ and define $Y_{\Theta} := \{x\in \mathbb S^{n-1} : M_{2d - 2}(x) = \Theta\}$. We then have 
$\varepsilon_2(\Theta)\leq M_{2d}(t) \leq \varepsilon_1(\Theta)$ for some $\varepsilon_1, \varepsilon_2$ as $t$ ranges over $Y_{\Theta}$.
Note that $\frac{1}{n^{d-1}} < \varepsilon_2 < \varepsilon_1 < 1$ since for $x\in \mathbb S^{n-1}$ $\frac{1}{n^{d-2}} < M_{2d - 2}(x) < 1$ implies $\frac{1}{n^{d-1}} < M_{2d}(x) < 1$. Now, choose some $v \in Y_{\Theta}$ such that $\varepsilon_2(\Theta) < M_{2d}(v) < \varepsilon_1(\Theta)$. Since we are dealing with even symmetric forms we can additionally assume w.l.o.g. that $v \in \mathbb S^{n-1}_+$. By Lemma \ref{Lem:Jacobi3} $v$ is not a $k$-point for $k\leq 2$. Hence, if $z \in \mathbb S^{n-1}_+$ is a $2$-point, then we claim (due to $v \in Y_{\Theta} \subset \mathbb S^{n-1}$) that
\begin{eqnarray*}
 p_v(z) &=&  (M_{2d}(z) - M_{2d}(v))^2 + (M_{2d-2}(z) - \Theta)^2 \ \geq \ \kappa\ >\ 0.
\end{eqnarray*}
For $M_{2d-2}(z) \neq \Theta$ this is obvious. If $M_{2d-2}(z) = \Theta$, then we use Lemma \ref{Lem:Jacobi3}. On the one hand, $\varphi(z)$ and $\varphi(v)$ can only differ in the last component. On the other hand, $z \in \partial \varphi(\mathbb R^3)$ and $v \notin \partial \varphi(\mathbb R^3)$. Thus, $M_{2d}(z) \neq M_{2d}(v)$.
Note that also $p_v(z) \geq\kappa > 0$ at all $2$-points $z \in \mathbb S^{n-1}\setminus \mathbb S^{n-1}_+$, since $p$ is even symmetric. 

Choosing $0 < \lambda < \kappa$ we conclude that
$p_{v,\lambda} := p_v - \lambda M_2^{2d}$ is nonnegative at all $2$-points but not nonnegative globally since $p_{v,\lambda}(v) = -\lambda < 0$. So, we have constructed a form  $p = \beta M_2^{6} + \gamma M_{6}^2 + \delta M_{6}M_2^3 + \alpha_1f_1 + \alpha_2f_2$ that is nonnegative at all $2$-points but not nonnegative globally and hence 
$M_{n,4d}^{(2)} = 1$, for $n \in \{d - 1, d\}.$
\end{proof}

We conclude the following corollary generalizing \cite[Theorem 3.3]{Harris}, which covers $n = 3$.
\begin{cor}
Let $4d \geq 12$. The set of $2$-points is not a test set for $\mathbb R[x_1,\dots,x_n]_{4d}^{S,e}$ for $n \leq d$.
\end{cor}

\begin{proof}
 Theorem \ref{thm:maximum} proves the corollary for $n \in \{d - 1, d\}$. For the general case we can multiply the form $p$ in Theorem \ref{thm:maximum} by an appropriate power sum in order to increase the degree.
\end{proof}

Another consequence of Theorem \ref{thm:maximum} is that for $n \in \{d - 1, d\}$ the set of all $n$-forms of degree $4d$ for which the set of $2$-points is a test set is not convex. For this, we recall that
\begin{eqnarray*}
A_{n,4d}^{(k)} &=& \{p \in \mathbb R[x_1,\dots,x_n]_{4d}^{S,e}\ : \ p\geq 0\Leftrightarrow p\geq 0\,\, \textrm{at all}\,\, k\textrm{-points}\}.
\end{eqnarray*}
Note that $A_{n,4d}^{(k)}\subseteq A_{n,4d}^{(k+1)}$ and $A_{n+1,4d}^{(k)}\subseteq A_{n,4d}^{(k)}$ for all $n, d, k\in \N$. Furthermore, by Timofte's theorem, we always have $A_{n,4d}^{(d)} = \mathbb R[x_1,\dots,x_n]_{4d}^{S,e}$ for $n \geq d$.

\begin{cor}
\label{cor:convex}
Let $4d \geq 12$ and $n\in\{d-1, d\}$. The set $A_{n,4d}^{(2)}$ is not convex.
\end{cor}

\begin{proof}
 By Theorem \ref{thm:maximum} there exists a form 
$$p \ = \  \beta M_2^{2d} + \gamma M_{2d}^2 + \delta M_{2d}M_2^d + \alpha_1f_1 +\alpha_2f_2 \,\notin\, A_{n, 4d}^{(2)}$$
for some $f_1, f_2 \in B\setminus\{M_2^{2d}, M_{2d}^2, M_{2d}M_2^d\}$. Obviously, $p = \frac{1}{2}p_1 + \frac{1}{2}p_2$ with $p_1 :=  \beta M_2^{2d} + \gamma M_{2d}^2 + \delta M_{2d}M_2^d + 2\alpha_1f_1$ and
$p_2 :=  \beta M_2^{2d} + \gamma M_{2d}^2 + \delta M_{2d}M_2^d + 2\alpha_2f_2$. By Theorem \ref{Thm:Main} we have $p_1, p_2 \in A_{n,4d}^{(2)}$.
\end{proof}

Notice that due to the inclusion $A_{n+1,4d}^{(k)}\subseteq A_{n,4d}^{(k)}$ it is not obvious that for $n < d - 1$ the upper corollary still holds.
We note furthermore, that the number $M_{n,4d}^{(2)}$ for $n > d$ seems to be more challenging to determine than for $n \leq d$. For example, for $n = 4$ and $4d = 12$ we have $\dim \mathbb R[x_1,x_2,x_3,x_4]_{12}^{S,e} = 9$. The Jacobian of the form $p  := \beta M_2^{6} + \gamma M_{6}^2 + \delta M_{6}M_2^3 + \alpha_1 M_8M_4$ does not satisfy Lemma \ref{Lem:Jacobi} and therefore it seems unclear whether $p$ is nonnegative if and only if it is nonnegative at all $2$-points. However, we conjecture that this is true as well as $M_{n,4d}^{(2)} = 1$ for $n > d$.

\section{Outlook}
Since  it seems a very difficult problem to determine conditions for forms in our investigated subspaces to be sums of squares, it would be an interesting task to analyze the difference between nonnegative forms and sums of squares in this case. Experimentally we were not able to construct nonnegative forms that are not sums of squares. Furthermore, we note that in the setting of Theorem \ref{Thm:Main2} the bound of $(m - 1)$ is not optimal in general. Consider the case
of nonnegative even symmetric octics in at least four variables, which is a $5$-dimensional convex cone. Following Theorem \ref{thm:Timofte} such forms are nonnegative if and only if
they are nonnegative at all $3$-points. But using Timofte's theorem we know that nonnegativity can be decided at $2$-points in this case. But
whenever the degree $4d$ is significantly larger than the number of variables, the bound of $(m - 1)$ components is significantly more useful. 
An interesting future prospect would be to analyze these bounds in an asymptotic sense.

Additionally, from a computational viewpoint it would be interesting to extend the experimental approach used in Section \ref{Sec:DimensionFour} in order to achieve more understanding of the semialgebraic structure of the cone of nonnegative even symmetric forms of degree $4d$ and its subcones. 

Maybe the most interesting follow-up task is to shed light at the numbers $M_{n,4d}^{(k)}$ for $k \geq 3$ as well as a deeper understanding of the geometrical and topological structure of the sets $A_{n,4d}^{(k)}$.

\section*{Acknowledgments}
We would like to thank Thorsten Theobald for his support during the development of this article.

\bibliographystyle{amsplain}
\bibliography{uniformbound}

\providecommand{\bysame}{\leavevmode\hbox to3em{\hrulefill}\thinspace}
\providecommand{\MR}{\relax\ifhmode\unskip\space\fi MR }
\providecommand{\MRhref}[2]{%
  \href{http://www.ams.org/mathscinet-getitem?mr=#1}{#2}
}
\providecommand{\href}[2]{#2}
\begin{thebibliography}{10}

\bibitem{Anai:Yanami}
H.~Anai and H.~Yanami, \emph{Sy{NRAC}: a {M}aple-package for solving real
  algebraic constraints}, Computational science---{ICCS} 2003. {P}art {I},
  Lecture Notes in Comput. Sci., vol. 2657, Springer, Berlin, 2003,
  pp.~828--837.

\bibitem{Blekherman:Volume}
G.~Blekherman, \emph{There are significantly more nonnegative polynomials than
  sums of squares}, Israel J. Math. \textbf{153} (2006), 355--380.

\bibitem{Blekherman:Parrilo:Thomas}
G.~Blekherman, P.A. Parrilo, and R.R. Thomas, \emph{Semidefinite optimization
  and convex algebraic geometry}, MOS-SIAM Series on Optimization, vol.~13,
  SIAM and the Mathematical Optimization Society, Philadelphia, 2013.

\bibitem{Blekherman:Riener}
G.~Blekherman and C.~Riener, \emph{Symmetric nonnegative forms and sums of
  squares}, 2012, Preprint, {\sf arXiv:1205.3102}.

\bibitem{Blum:et:al:np}
L.~Blum, F.~Cucker, M.~Shub, and S.~Smale, \emph{Complexity and real
  computation}, Springer-Verlag, New York, 1998.

\bibitem{Choi:Lam:Reznick}
M.D. Choi, T.Y. Lam, and B.~Reznick, \emph{Even symmetric sextics}, Math. Z.
  \textbf{195} (1987), no.~4, 559--580.

\bibitem{Gatermann:Parrilo}
K.~Gatermann and P.A. Parrilo, \emph{Symmetry groups, semidefinite programs,
  and sums of squares}, J. Pure Appl. Algebra \textbf{192} (2004), no.~1-3,
  95--128.

\bibitem{Harris:Diss}
W.R. Harris, \emph{Real even symmetric forms}, ProQuest LLC, Ann Arbor, MI,
  1992, Thesis (Ph.D.)--University of Illinois at Urbana-Champaign.

\bibitem{Harris}
\bysame, \emph{Real even symmetric ternary forms}, J. Algebra \textbf{222}
  (1999), no.~1, 204--245.

\bibitem{Hilbert:Seminal}
D.~Hilbert, \emph{Ueber die {D}arstellung definiter {F}ormen als {S}umme von
  {F}ormenquadraten}, Math. Ann. \textbf{32} (1888), no.~3, 342--350.

\bibitem{Lasserre}
J.B. Lasserre, \emph{Moments, positive polynomials and their applications},
  Imperial College Press Optimization Series, vol.~1, Imperial College Press,
  London, 2010.

\bibitem{Laurent:Survey}
M.~Laurent, \emph{Sums of squares, moment matrices and optimization over
  polynomials}, Emerging applications of algebraic geometry, IMA Vol. Math.
  Appl., vol. 149, Springer, New York, 2009, pp.~157--270.

\bibitem{Parrilo:SOSTOOLS}
S.~Prajna, A.~Papachristodoulou, P.~Seiler, and P.A. Parrilo, \emph{S{OSTOOLS}
  and its control applications}, Positive polynomials in control, Lecture Notes
  in Control and Inform. Sci., vol. 312, Springer, Berlin, 2005, pp.~273--292.

\bibitem{Riener:Diss}
C.~Riener, \emph{Symmetries in semidefinite and polynomial optimization}, Ph.D.
  thesis, Goethe University Frankfurt am Main, 2011.

\bibitem{Riener}
\bysame, \emph{On the degree and half-degree principle for symmetric
  polynomials}, J. Pure Appl. Algebra \textbf{216} (2012), no.~4, 850--856.

\bibitem{Sagan}
B.E. Sagan, \emph{The symmetric group}, second ed., Graduate Texts in
  Mathematics, vol. 203, Springer-Verlag, New York, 2001.

\bibitem{Timofte}
V.~Timofte, \emph{On the positivity of symmetric polynomial functions. {I}.
  {G}eneral results}, J. Math. Anal. Appl. \textbf{284} (2003), no.~1,
  174--190.

\end{thebibliography}

\end{document}